\documentclass [12pt]{article}
\usepackage{amsmath,amssymb}
\usepackage{amssymb,verbatim}

\usepackage{blkarray}
\usepackage{amsmath}
\usepackage{tikz-cd} 

\usepackage{faktor}
\usepackage{xfrac}
\usepackage{faktor} 
\usepackage{amsmath} 
\usepackage{amssymb}
\usepackage{mathtools}

\usepackage[utf8]{inputenc}
\usepackage{graphicx}
\usepackage{graphicx,amssymb,amstext,amsmath}
\usepackage{tikz}
\usepackage{amsmath}
\usepackage{amssymb}
\usepackage{amsthm}
\usepackage{hebcal}
\usetikzlibrary{arrows, decorations.markings}
\usepackage{blindtext}
\usepackage{amsthm}
\usepackage[T1]{fontenc}
\usepackage[utf8]{inputenc}
\usepackage{blindtext}
\usepackage[T1]{fontenc}
\usepackage[utf8]{inputenc}
\newcommand{\interior}[1]{%
  {\kern0pt#1}^{\mathrm{o}}%
}
\usepackage{float}
\usepackage{enumitem}
\usepackage{multirow}
\usepackage{array}
\newcolumntype{L}[1]{>{\raggedright\let\newline\\\arraybackslash\hspace{0pt}}m{#1}}
\newcolumntype{C}[1]{>{\centering\let\newline\\\arraybackslash\hspace{0pt}}m{#1}}
\newcolumntype{R}[1]{>{\raggedleft\let\newline\\\arraybackslash\hspace{0pt}}m{#1}}
\usepackage{tikz}
\usetikzlibrary{positioning, arrows.meta}

\usepackage{cite}
\newtheorem{thm}{Theorem}[section]
\newtheorem*{thm*}{Theorem}
\newtheorem{lemma}[thm]{Lemma}
\newtheorem{prop}[thm]{Proposition}
\newtheorem{cor}[thm]{Corollary}
\newtheorem{definition}[thm]{Definition}

\newtheorem{remark}[thm]{Remark}

\newtheorem{question
}{Question}

\def\0{{\bf 0}}

\def\Z{{\bf Z}}

\def\vN{\mathbb{N}}
\def\Z{\mathbb{Z}}

\def\vR{\mathbb{R}}

\def\OO{\mathcal{O}}

\def\vE{\mathbb{E}}

\def\p{\mathfrak{p}}
\def\vol{{\rm vol}}

\def\vol{{\rm vol}}
\makeatletter
\def\keywords{\xdef\@thefnmark{}\@footnotetext}
\title{Simultaneous visibility in the algebraic lattice}
\date{}
\author{
Rishi Kumar\footnote{Department of Mathematics, Tel Aviv University, Israel. E-mail: rkumar@tauex.tau.ac.il}
\and
Wataru Takeda\footnote{Department of Mathematics, Toho University, Japan. E-mail: wataru.takeda@sci.toho-u.ac.jp}}
\begin{document}
\maketitle
\keywords{2020 \bf{Mathematics Subject Classification:} Primary 11P11, 11R45, Secondary 11P21, 11R04.}
\keywords{\bf{Key words and phrases. Number fields, algebraic integers, natural density, generalized cut-and-project sets}}
\begin{abstract}
Let $K$ be a number field with ring of integers $\OO$. Two lattice points ${\bf x, y}\in \OO^m$ with $m\geq 2$ are said to be visible from one another if $\gcd((x_i-y_i),\ldots, (x_m-y_m))=\OO$, where $(x_i-y_i)$ is the ideal generated by $x_i-y_i$. Let $S\subset \OO^m$ be a finite set. For $K=\mathbb{Q}$, the asymptotic density of the set of lattice points, visible from all points of $S$, was studied by several authors. For general number fields $K$, however, the asymptotic density has been studied only in the special case $S=\{(0,\ldots,0)\}$. Our main result establishes the corresponding density formula for a number field $K$ whose ring of integers $\OO$ is a principal ideal domain, for all finite sets $S$ with $|S|\geq 2$.
\end{abstract}

\section{Introduction and statement of results}
The distribution of lattice points is a classical subject with a rich history. Building on this foundation, this work investigates the asymptotic density of subsets within $\OO^m$ characterized by simultaneous visibility conditions, where $\OO$ represents the ring of integers of a number field $K$.
We begin by stating our main result and then review the relevant background and previous work.

\subsection{Main results}
We say two $m$-tuple $\textbf{a}=(a_1,\ldots, a_m), \textbf{b}=(b_1,\ldots, b_m)\in \OO^m$ are mutual visible from each other, if $\gcd(\textbf{a}-\textbf{b})=\gcd((a_1-b_1), \ldots,(a_m- b_m))= \OO$, where $(a_i-b_i)$ is the ideal generated by $a_i-b_i$.

For a finite set $S\subseteq \OO^m$ we define
\[
V(S)= \{{\bf z}\in \OO^m : \gcd({\bf z-x})= \OO \text{ for all } {\bf x}\in S\}.
\]
Thus $V(S)$ is the set of points that are simultaneously visible from all elements of $S$.

Let $\mathcal{P}_{\OO}$ be the set of prime ideals of $\OO$. 
For each $\p \in \mathcal{P}_\OO$, denote by
$$\pi_{\p}: \OO^m\to (\OO/\p)^m$$
the natural projection. Throughout the paper, we denote $|A|$, to be the cardinality of the set $A$.
The following theorem determines the density (see Section \ref{Section: visibility in number fields}) of $V(S)$.
\begin{thm*}[Theorem \ref{thm: main theorem statments}]
Let $S\subseteq \OO^m$ be a finite set. Then
\[
D(V(S))= \prod_{\p\subset \OO,\, {\rm prime}}
\left(1-\frac{|\pi_{\p}(S)|}{N(\p)^m}\right).
\]
\end{thm*}

This result generalizes the classical density formula for visible lattice points in $\Z^m$ and extends previous work on visibility problems for algebraic integers as we will explain in the following subsections.

\subsection{Visibility from the origin in $\mathbb{Z}^m$}
We begin by recalling the definition of visibility from the origin in the $m$-dimensional integer lattice $\mathbb{Z}^m$.

Let $m\geq 2$. Two distinct points $\textbf{x}$ = $(x_1,\ldots ,x_m)$ and $\textbf{y}$ = $(y_1, \ldots ,y_m)$ in $\Z^m$ are said to be \textit{mutually visible} if no other lattice point lies on the line segment joining them. Equivalently, 
$$\gcd(x_1- y_1, \ldots, x_m -y_m) =1.$$
Let $V_m$ denote the set of lattice points in $\mathbb{Z}^m$ that are visible from the origin. The asymptotic density of a set $A\subseteq \mathbb{Z}^m$, measured along the boxes $[-L,L]^m$, is defined by
$$D(A) = \lim_{L\to \infty} \frac{\left|A\cap [-L,L]^m\right|}{\vol([-L,L])^m},$$
(provided the limit exists).
In 1849, Dirichlet proved that $D(V_2)= \frac{1}{\zeta(2)}$ and established an error bound:
\begin{equation}\label{eq. error term for m equal 2}
\frac{|V_2\cap[-L,L]^2 |}{\vol([-L,L])^2}= \frac{1}{\zeta(2)} + O\left(\frac{L^{\delta}}{L^2}\right),\qquad 1<\delta<2.    
\end{equation}
The error term in \eqref{eq. error term for m equal 2} has been improved by F. Mertens \cite{MERF} in 1874 and later by A. Walfisz \cite{WALFISZ} in 1959. The current best known bound is $$ O\left(\frac{(\log L)^{2/3}(\log \log L)^{1/3}}{L}\right)$$ established by Liu \cite{LIU} in 2016.

In 1900, Lehmer~\cite{DNL} extended Dirichlet's result to any dimension $m\geq 3$, and proved that 
$$D(V_m)= \frac{1}{\zeta(m)}.$$ 
Nymann \cite{Nym} in 1972 established an error bound
\begin{equation}\label{eq: error term for Vm with m greater than 2}
\frac{| V_m \cap [-L,L]^m|}{\vol([-L,L])^m}= \frac{1}{\zeta(m)} + O\left(\frac{1}{L}\right).    
\end{equation} 
The error bound in \eqref{eq: error term for Vm with m greater than 2} is the sharp bound (see \cite{TAKEDA, Berend1}). Baake, Moody, and Pleasants \cite{Moody} studied the density of the set of visible points in the setting of more general lattices and for different choices of averaging sets, including Euclidean balls  (see also the appendix of \cite{Moody} for a broader discussion of more general averaging sets). They proved in \cite[Proposition 6]{Moody} that for large $R$
\begin{equation}\label{eq. density along balls for Vk}
\begin{split}
\frac{|V_m\cap B_R|}{\vol(B_R)}&= \frac{1}{\zeta(m)} + \begin{cases}
    O\left(\frac{\log R}{R}\right),&\qquad m=2,\\
    O\left(\frac{1}{R}\right),&\qquad m\geq 3,
\end{cases}   
\end{split}    
\end{equation}
where $B_R\subset \mathbb{R}^m$ is ball of radius $R$ center at the origin.
In contrast, to the confirmed sharpness of the error bound for averaging over cubes $[-L, L]^m$ in equation \eqref{eq: error term for Vm with m greater than 2}, it is currently unknown whether the bound for $m \ge 3$ in equation \eqref{eq. density along balls for Vk} is sharp.

Recent work in \cite{ACZ} has explored visibility phenomena in higher dimensions and shown that almost all self-visible triangles with vertices in $[0,N]^d$ are nearly equilateral, with side lengths close to $\frac{N\sqrt{d}}{\sqrt{6}}$. 
The orchard visibility problem is known as a similar problem. This problem considers lattice points as objects of positive size, and two points are considered visible if the line of sight between them is not obstructed, reflecting the intuitive picture of trees in an orchard (see \cite{Kru}).

\subsection{Simultaneous visibility in $\mathbb{Z}^m$}
This subsection recalls the definition of simultaneous visibility in $\mathbb{Z}^m$ from a finite set of integer points.

A set $S$ is \textit{admissible} if any two distinct elements of $S$ are mutually visible. It can easily be seen that if $S$ is admissible, then $|S|\leq 2^m$. In the case of admissible set $S$, Rearick~\cite{ReaM} showed that
\begin{equation}\label{admi}
\begin{split}  
\frac{|V(S)\cap [-L,L]^m|}{\vol([-L,L])^m}&= \prod_{p\in \mathcal{P}}\left(1- \frac{|S|}{p^m}\right) + \begin{cases} 
 O\left(\frac{1}{L}\right), & \qquad   |S|< m-1, \\
O\left(L^{- \frac{m-1}{|S|}+\varepsilon}\right), ~\forall~\varepsilon > 0, & \qquad |S|\geq m-1,
\end{cases}
\end{split}
\end{equation}
where $\mathcal{P}$ is the set of all primes. Liu, Lu, and Meng \cite{LU} investigated simultaneous visibility in the setting of curves. When restricted to straight-line visibility, their results yield improvements to Rearick’s error bound in certain cases.

Rumsey~\cite{Rum} considered the more general situation in which $S$ is an arbitrary subset of $\Z^m$. For each prime $p \in \mathcal{P}$, let
$$\pi_p: \Z^m \to (\Z/ p\Z)^m$$ 
denote the natural projection and define
\begin{equation}\label{S(p)}
s(p) = |\pi_p(S)|, \qquad p\in \mathcal{P}.
\end{equation}
When $S$ is finite set, Rumsey~\cite{Rum} proved that the density (along cubes $[-L,L]^m$) of points in $\mathbb{Z}^k$ simultaneous visible from $S$ is given by
\begin{equation}\label{s(p)}
D(V(S)) = \prod_{p\in \mathcal{P}}\left(1 - \frac{s(p)}{p^m}\right).
\end{equation}
Moreover, he extended this formula to certain infinite subsets $S$, satisfying appropriate conditions. He remarked that the density formula also holds when averaging over growing boxes $[-L_1, L_1]\times \cdots \times [-L_m, L_m]$, as $\min\{L_1,\ldots, L_m\}\to \infty$.

Rumsey's approach of simultaneous visibility of an arbitrary finite set $S \subset \Z^m$ has been further investigated in any large box located anywhere (see \cite[Theorem 1.1]{Berend1}). 
In particular, when boxes are $[-L,L]^m$, it has been shown \cite[Theorem 1.1]{Berend1} that as $L\to \infty$,
\begin{equation}\label{equation on upper bound in main theorem}
\frac{|V(S)\cap [1,L]^m|}{\vol([-L,L])^m} \leq \prod_{p \in \mathcal{P}}\left(1 - \frac{s(p)}{p^m}\right) + \begin{cases} 
 O\left(L^{-2/3+\varepsilon}\right),\forall \varepsilon>0, & \qquad m = 2,\\
     O\left(\frac{1}{L}\right), &\qquad    m\geq 3.\\
    \end{cases}
 \end{equation}
It has been remarked in \cite{Berend1} that the inequality in \eqref{equation on upper bound in main theorem} can be reversed, i.e., 

\begin{equation}\label{equation on equality in main theorem}
\frac{|V(S)\cap [1,L]^m|}{\vol([-L,L])^m} = \prod_{p \in \mathcal{P}}\left(1 - \frac{s(p)}{p^m}\right) + \begin{cases} 
 O\left(L^{-2/3+\varepsilon}\right),\forall \varepsilon>0, & \qquad m = 2,\\
     O\left(\frac{1}{L}\right), &\qquad    m\geq 3.\\
    \end{cases}
 \end{equation}

\begin{prop}
Let $S$ be any finite set, and let $V(S)$ be the set of lattice points in $\mathbb{Z}^m$ simultaneously visible from $S$. Then the density of $V(S)$, when we take averaging  along balls, is given by
$$\lim_{R\to \infty}\frac{|V(S)\cap B_R |}{\vol(B_R)}= \prod_{p\in \mathcal{P}}\left(1 - \frac{s(p)}{p^m}\right),$$
where $B_R$ is the ball of radius $R$ with center the origin.
\end{prop}

The error term in \eqref{equation on equality in main theorem} if confirmed, would allow the error bound to be used in proving the above proposition. Here, the proof of the proposition is a direct consequence of Theorem \ref{thm: main theorem statments} in the case where $K = \mathbb{Q}$ and $e_i$ represents the coordinate basis.

\subsection{Simultaneous visibility in $\mathcal{O}^m$}\label{Section: visibility in number fields}
Let $K$ be a number field with $n=[K:\mathbb{Q}]$ and $\OO$ be its ring of integers. 
As a generalization to $\OO$, several authors considered visibility in $\mathcal{O}^m$.
There are two major types of formulations.

In the first one, we regard a tuple $(I_1,\ldots, I_m)$ of ideals as a lattice point of $K^m$. 
\begin{definition}\emph{
Let $I_1,\ldots,I_m\subseteq \OO$ be ideals. The greatest comman divisor $\gcd(I_1,\ldots, I_m)$ is defined to be the ideal $\mathcal{A}\subseteq \OO$ which satisfies:
\begin{itemize}
\item $I_i\subseteq \mathcal{A}$ for all $1\leq i\leq m$.
\item If there exists some ideal $\mathcal{B}\subseteq \OO$ such that $I_i\subseteq \mathcal{B}$, then $\mathcal{A}\subseteq \mathcal{B}$.
\end{itemize}
}
\end{definition}
An $m$-tuple $(I_1,\cdots, I_m)$ of ideals is visible from the origin, if $\gcd(I_1,\cdots, I_m)= \OO$. For an ideal $I\in \OO$, the {\em norm} of $I$ is defined as
$$N(I)= \left|\OO/I\right|,$$
which is always finite. 
Let $V_{m,{\rm ideal}}(x)$ denote the number of visible lattice points $(I_1,\cdots, I_m)$ from the origin with $N(I_j)\leq x$. 
Then, B. D. Sittinger \cite{Sittinger}, proved that
\begin{equation}
\begin{split}
V_{m,{\rm ideal}}(x)&= \frac{c^m}{\zeta_K(m)}x^m + \begin{cases}
    O(x^{m-1/n}),&\quad m\geq 3\\
    O(x^{2-1/n} \log x),&\quad m=2,
\end{cases}
\end{split}    
\end{equation}
where
$$\zeta_{K}(s)= \sum_{I\subseteq \OO}\frac{1}{N(I)^s},\qquad \Re(s)>1,$$
and $c$ being a positive constant depending only on $K$.
In \cite{Takeda}, the Lindel\"of Hypothesis for Dedekind zeta function implies that 
\[V_{m,{\rm ideal}}(x)= \frac{c^m}{\zeta_K(m)}x^m +O(x^{m-1/2+\varepsilon}).\]
The M\"obius function $\mu$ is defined for ideals $I\subset \OO$ as follows:
\begin{equation*}
\begin{split}
\mu (I)&= \begin{cases}
    1,&\qquad N(I)=1,\\
    (-1)^r,& \qquad I= \p_1\cdots \p_r,\, \textup{for distinct prime ideals}\, \p_1,\ldots,\p_r,\\
    0, & \qquad \p^2\, |\,  I \, \textup{for some prime ideal}\, \p.
\end{cases}    
\end{split}    
\end{equation*}
Using the M\"obius function, the reciprocal of the Dedekind zeta
function can be expressed as 
\begin{equation}\label{equation of zeta(K, d)}
\frac{1}{\zeta_{K}(s)} = \sum_{I\subset
  \OO}\frac{\mu(I)}{N(I)^s}=\prod_{\p\subset \OO,\, {\rm{prime}}}\left(1-\frac{1}{N(\p)^s}\right),  
\end{equation}
and these series and product converge absolutely when $\Re(s)>1$. 
 
The second formulation, we consider tuples of algebraic integers rather than tuples of ideals. 
Accordingly, we assume that $K$ is a number field whose ring of integers $\mathcal{O}_K=\OO$ is a principal ideal domain.
In this paper, we use this formulation.

Let $\vE=\{e_i\}_{i=1}^n$ be a $\mathbb{Z}$-basis for $\OO$. Define
\begin{equation}\label{eq: sets C[L,E]}
C[L, \vE]=\left\{\sum_{i=1}^n a_i e_i\, :\, a_i\in [-L, L]\cap \Z\right\},\qquad L\in \mathbb{N},    
\end{equation}
and
\begin{equation}\label{eq: sets B[R,E]}
B[R, \vE]=\left\{\sum_{i=1}^n a_i e_i\, :\, a_i\in \Z,\,  a_1^2+\cdots + a_n^2\leq R^2\right\},\qquad R\in \mathbb{R}_{+}.    
\end{equation}

Let $T\subset \OO^m$ , and let $(A_n)_{n\in \mathbb{N}}$ be any sequence of subsets of $\OO$ as defined in \eqref{eq: sets C[L,E]} an \eqref{eq: sets B[R,E]}. We define the upper density and lower density of $T$ with respect to $\vE$ as
$$\overline D_{\vE}(T)= \limsup_{n\to \infty}\frac{|A_n^m\cap T|}{(\vol(A_n))^{m}},\qquad \underline D_{\vE}(T)= \liminf_{n\to \infty}\frac{|A_n^m\cap T|}{(\vol(A_n))^{m}},$$
respectively. 
We say that $T$ has density $d$ with respect to $\vE$ if $\overline D_{\vE}(T)= \underline D_{\vE}(T)$, in which case one sets
$$D_{\vE}(T)=\overline D_{\vE}(T)= \underline D_{\vE}(T)=d.$$
It is easy to see that the density may depend on the chosen basis (see \cite[Example 2.2]{Ferraguti}). Whenever this density is independent of the chosen basis $\vE$, it is consistent to denote the density of the set $T$ by $ D(T)$ without any subscript.

Ferraguti and Micheli proved that the density of the set of coprime  $m$-tuples of algebraic integers is $1/\zeta_K(m)$ \cite[Theorem 3.7]{Ferraguti}. After that, Sittinger developed a further refinement in \cite{Sittinger2}.
Instead of requiring that all $m$ algebraic integers be relatively $r$-prime, Sittinger studied the weaker condition that every subcollection of $j$ elements is relatively $r$-prime.

Let $S\subseteq \OO^m$ be a finite set, we continue to denote $V(S)$ as in the case of $\Z^m$ the set of points of $\OO^m$, visible simultaneously from all points of $S$, i.e.,
$$V(S)= \{{\bf z}\in \OO^m\, :\, \gcd({\bf z-x})= \OO, \forall\,{\bf x}\in S\}.$$
Let $\mathcal{P}_{\OO}$ be the set of all prime ideals in $\OO$. For each prime $\p \in \mathcal{P}_\OO$, denote $\pi_{\p}$ be the the natural projection
$$\pi_{\p}: \OO^m\to \left(\OO/\p\right)^m,$$
and 
$$s(\p)= |\pi_{\p}(S)|.$$

\begin{thm}\label{thm: main theorem statments}
Let $S\subseteq \OO^m$ be a finite set. Then the density of $V(S)$ along an averaging sequence of $C[L,\vE]$ or $B[R,\vE]$ exists, and is given by:
$$ D(V(S))= \prod_{\p\in \mathcal{P}_{\OO}}\left(1-\frac{s(\p)}{N(\p)^m}\right).$$
\end{thm}

\begin{remark}
    As in \cite[Theorem 3.7]{Ferraguti}, the density $ D(V(S))$ is independent of the chosen basis $\vE$.
\end{remark}

Similar to the integer lattice case, we call a finite set $S$ admissible if any two distinct elements
of $S$ are mutually visible. In this case $|S|\leq p^{m}$, where $p=\min\{N(\mathfrak{p}):\mathfrak{p}\subset\OO\text{ is prime ideal}\}$.

\begin{cor}
    Let $S$ be an admissible set of $\OO^m$. Then 
    \[ D(V(S))= \prod_{\p\in \mathcal{P}_{\OO}}\left(1-\frac{|S|}{N(\p)^m}\right).\]
\end{cor}

\subsection{Acknowledgements}
The authors are grateful to Daniel Berend for many helpful discussions, and to Faustin Adiceam, Carlos Ospina, Zeev Rudnick, and Barak Wiess for useful comments. 
The first author's work is supported by the  grant
ISF 2860/24.
The second author was supported by Grant-in-Aid for Early-Career Scientists (Grant Number: JP22K13900).
\section{Preliminary}

In this section, we recall the necessary preliminaries for the proofs in Section \ref{sec: proofs}: Van Hove sequences and generalized cut-and-project sets.
\subsection{Van Hove sequences} Let $G$ be a $\sigma$-compact be Locally Compact Abelian groups (LCAG) and $\theta$ be the Haar measure on $G$. 
An \emph{averaging sequence} $(A_{n})^{}_{n\in\mathbb{N}}$ is a sequence of compact sets $A_{n}$ with the property that $A_{n} \subset A_{n+1}$ for every $n\in\mathbb{N}$ and
$\bigcup_{n\in\mathbb{N}} A_{n} = G$. The sequence $(A_{n})^{}_{n\in\mathbb{N}}$ is called \emph{van Hove} if, for every compact $K\subset G$,
\[
     \lim_{n\to\infty} \frac{\theta (\partial^{K}\! A_{n})}
     {\theta (A_{n})} \, = \, 0.
\]
Here, for any compact set $B\subset G$, the $K$-boundary $\partial^{K}\! B$ is defined by
$$\partial^{K}\! B =
\bigl( (K+ B)\setminus \interior{B}\bigr) \cup
\bigl((-K+ B^{\mathsf{c}})\cap B\, \bigr),$$
where
$B^{\mathsf{c}}$ is the complement of $B$ in $G$ and $\interior{B}$ its interior (see \cite[p.125]{Moody1}). For example when $G=\mathbb{R}^d$, then the sequence of sets $A_n= [-n,n]^d$ or $A_n= B_n$, are van-Hove sequence, where $B_n$ is the ball of radius $n$ with center at the origin. The existence of van Hove sequences in
$\sigma$-compact LCAGs is shown in \cite{Martin}.
The density of a point set $\Lambda\subset G$, along the averaging sequences $(A_n)_{n\in \mathbb{N}}$ is defined as:
\[
    D (\Lambda) \, = \lim_{n\to\infty}
    \frac{|\Lambda \cap A_{n}| }{\theta (A_{n})}  ,
\]
provided the limit exists. The lower and upper densities are defined as

\begin{equation}\label{eq:lower-upper-dens}
   \underline{D} (\Lambda) \, = \, \liminf_{n\to\infty}
   \frac{|\Lambda \cap A_{n}| }{\theta(A_{n})}
   \quad \text{and} \quad
   \overline{D} (\Lambda) \, = \, \limsup_{n\to\infty}
   \frac{|\Lambda \cap A_{n}| }{\theta (A_{n})}  ,
\end{equation}
which always exist, with $0\leqslant\underline{D} (\Lambda)
\leqslant\overline{D} (\Lambda)\leqslant\infty$.
It can be easily checked that the sequence sets $C[L,\vE]$ and $B[R,\vE] $ form a van Hove sequence.

\subsection{Generalized cut-and-project sets}
Let $G$ and $H$ be locally compact Abelian groups. Denote their Haar measures by $\theta_G$ and $\theta_H$, respectively. We further assume that $G$ is $\sigma$-compact and $H$ is compactly generated. 


A cut-and-project scheme consists of the product group $ G\times H$ together with a lattice $\mathcal{L}\subset G\times H$ (i.e. a discrete cocompact subgroup) and the projections 
$$\pi_{\textup{phys}}: G\times H\to G,\quad {\rm and} \quad \pi_{\textup{int}}:G\times H\to H,$$
called physical and internal projection, respectively. 
We will assume that $\pi_{\textup{phys}}|_{\mathcal{L}}$ is injective and  $\pi_{\textup{int}}(\mathcal{L})$ is dense in $H$.
A cut-and-project set is any set of the form
$$\Lambda(\mathcal{W},\mathcal{L})= \{\pi_{\textup{phys}}(y)\, : \,
y\in \mathcal{L}, \, \pi_{\textup{int}}(y)\in \mathcal{W}\},$$
with $\mathcal{W}\subset H$. We will assume that $\mathcal{W}$ to be relatively compact. When $\varnothing\neq \mathcal{W}= \overline{\interior{\mathcal{W}}}$ is compact, the window is called proper. When in addition $\theta_{H}(\partial \mathcal{W})=0$, the cut-and-project set is called {\em irreducible}. The density of irreducible cut-and-project set $\Lambda(\mathcal{W},\mathcal{L})$ along a van-Hove sequence $\mathcal{A}=(A_n)_{n\in \mathbb{N}}$ is given by
\begin{equation}\label{eq: the density formula of irr CPS}
D({\Lambda(\mathcal{W},\mathcal{L})})=   \lim_{n\to\infty}
    \frac{|\Lambda(\mathcal{W},\mathcal{L}) \cap A_{n}| }{\theta_G (A_{n})}= D(\mathcal{L})\cdot \theta_H(\mathcal{W}) 
\end{equation}
(see \cite[Theorem 1]{Moody1}). We refer to \cite{Meyer} for more details on cut-and-project sets. 

\section{Proof of Theorem \ref{thm: main theorem statments}}\label{sec: proofs}
Fix a number field $K$ and its ring of integer $\OO_K$. For rational primes $p$, we denote by $M_p=\{\p_{1}^{(p)}, \ldots, \p_{\lambda_p}^{(p)}\}$
the set of distinct prime ideals lying above $p$, where $\lambda_p$ is the total number of distinct prime ideal lying over $\p$.
Denote
$$\vE_{t}= \bigcup_{i=1}^tM_{p_i},\qquad t\in \mathbb{N},$$
where $p_i$ is the $i$-th rational prime.
Let
$$V(S)_{\vE_t}=\{{\bf z}\in \OO^m\, :\, \p\not| \gcd({\bf z-s}),\, \forall \p \in \vE_t,\, \forall {\bf s}\in S\}$$
be the set of points which are simultaneous $\p$-{\it visible} for all $\p\in \vE_t$. In the following lemma, we show that the density of the set $\vE_{t}$ along averaging sets $C[L,\vE]$ (or $B[R,\vE]$) exists and is independent of the chosen basis.

\begin{lemma}\label{lem: for density upto finite primes}
Let $S\subseteq \OO^m$ be a finite set. Then
$$ D(V(S)_{\vE_t})= \prod_{\p\in \vE_t}\left(1-\frac{s(\p)}{N(\p)^m}\right).$$
\end{lemma}
\begin{proof}
Let $G=\OO^m$ (the physical space) equipped with the discrete topology and the counting Haar measure (so the Haar measure of any finite subset $A\subset G$ is $|A|$). Set the internal space
$$H= \prod_{\p\in \vE_t}\left(\OO/\p\right)^m,$$ 
and we equip $H$ with its normalized Haar probability measure $\theta_H$.
Define the diagonal embedding 
$$i: \OO^m\hookrightarrow G\times H,\quad (x_1,\ldots,x_m) \hookrightarrow ((x_1,\ldots,x_m), (x_1\, {\rm mod}\, \p,\ldots,x_m\, {\rm mod}\, \p)_{\p\in \vE_t}).$$
Write $\mathcal{L}'=i(\OO^m)$. Then, $\mathcal{L}'\subset G\times H$ is a lattice. Clearly, $\pi_G|_{\mathcal{L}'}$ is injective and $\pi_H|_{\mathcal{L}'}$ is dense. Thus the standard cut-and-project setup applies to $(G,H,\mathcal{L}')$. The window $\mathcal{W}'\subseteq H$ is defined by removing the residue classes coming from $S$:
$$\mathcal{W}'= \prod_{\p\in \vE_t}\left(\left(\OO/\p\right)^m\setminus\pi_{\p}(S)\right).$$
Each factor $\left(\OO/\p\right)^m\setminus \pi_{\p}(S)$ is a nonempty finite subset of $(\OO\setminus \p)^m$. Therefore, $\mathcal{W}'$ is compact, has nonempty interior (a product of nonempty open sets), and its boundary has Haar measure zero. Hence, $\mathcal{W}'$ is a proper window. By construction, the associated cut-and-project set
$$\Lambda'(\mathcal{W}',\mathcal{L}')= \{\pi_{G}(l)\, :\, l\in \mathcal{L}', \pi_{H}(l)\in \mathcal{W}'\}$$
is irreducible and consists exactly of those ${\bf x}\in \OO^m$ whose reduction modulo every prime $\p\in \vE_t$ avoids the residue classes $\pi_{\p}(S)$; that is,
$$\Lambda'(\mathcal{W}',\mathcal{L}')= V(S)_{\vE_t}.$$
 
Next, compute the Haar measure of the window $\mathcal{W}'$. Because $H$ is the product of the finite factors and $\theta_H$ is the product probability measure,
\begin{equation*}
\begin{split}
\theta_H(\mathcal{W}')&= \prod_{\p\in \vE_t}\frac{|(\OO\setminus \p)^m\setminus \pi_{\p}(S)|}{|\OO\setminus \p|^m}= \prod_{\p\in \vE_t}\left(1-\frac{s(\p)}{N(\p)^m}\right).
\end{split}    
\end{equation*}
It can easily be seen that the density of the lattices $\mathcal{L}'$ is equal to 1 along any van-Hove sequence. 
Since $\Lambda(\mathcal{W}',\mathcal{L}')= V(S)_{\vE_t}$, applying the standard density formula \eqref{eq: the density formula of irr CPS} for irreducible cut-and-project sets, we get
$$D(V(S)_{\vE_t})= D({\Lambda(\mathcal{W}',\mathcal{L}')})= D{(\mathcal{L}')}\cdot \theta_H(\mathcal{W}') = \prod_{\p\in \vE_t}\left(1-\frac{s(\p)}{|N(p)|^m}\right),$$
which completes the proof of the lemma.
\end{proof}

\begin{remark}
\emph{It follows from the proof of the lemma that the density of the set $V(S)_{\vE_t}$ exists along any van Hove sequence in $\OO^m$.}
\end{remark}

\subsection{Proof of Theorem \ref{thm: main theorem statments}}
Since $V(S)\subseteq V(S)_{\vE_t}$, using Lemma \ref{lem: for density upto finite primes}, we get for every natural number $t$,
\begin{equation}\label{eq: upper density of V(S)}
\overline D_{\vE}(V(S))\leq \overline D_{\vE}(V(S)_{\vE_t})=  D(V(S)_{\vE_t})=\prod_{\p\in \vE_t}\left(1-\frac{s(\p)}{|N(p)|^m}\right).    
\end{equation}
By \eqref{eq: upper density of V(S)} and letting $t$ run to infinite gives the inequality
\begin{equation}\label{eq: upper density of V(S) 2}
\overline D_{\vE}(V(S))\leq \prod_{\p\in \mathcal{P}_{\OO}}\left(1-\frac{s(\p)}{N(\p)^m}\right).    
\end{equation}
To get the reverse inequality (so that the two-sided density exists and equals the infinite product), note the following simple estimate
$$ D(V(S)_{\vE_t})-\overline D_{\vE}\left(V(S)_{\vE_t}\setminus V(S)\right)\leq \underline D_{\vE}(V(S)).$$
Therefore, to show the lower density of $V(S)$ approaches the same product as the upper density, it is enough to prove that the upper density $\overline D_{\vE}\left(V(S)_{\vE_t}\setminus V(S)\right)$ tends to zero as $t\to \infty$.

For convenience, we introduce a short notation about primes. For a prime ideal $\p\subset \OO$, for the $t$-th prime number $p_t$ and for an integer $M$, write $\p\succ M$ to mean that $\p$ lies above some rational prime greater than $M$. Dually, write $M\succ \p$ to mean the rational prime lying under $\p$ is less than $M$. With this notation note that if $\p \succ p_t$, then $\p +(p_i)=\OO$ for every $i\leq t$.

We claim that
\begin{equation}\label{eq. non visible from prime larger than pt}
V(S)_{\vE_t}\setminus V(S)\subseteq\bigcup_{{\bf s}\in S}\bigcup_{\p\in \mathcal{P}_{\OO}: \; \p\succ p_t} ({\bf s}+\p^m)\subseteq \OO^m    
\end{equation}
where $\p^m= \{(x_1,\ldots, x_m)\in \OO^m\, :\, x_i\in \p\}$ and ${\bf s} + \p^m= \{{\bf s}+ {\bf u}\,:\, u\in \p^m\}$. In fact, 
for $x\in V(S)_{\vE_t}\setminus V(S)$, we have
$$\forall {\bf s} \in S, \quad {\rm all\, prime\,\,  ideal} \, \p \, {\rm above}\, p_i\, {\rm with}\,  i\leq t,\, \p\not|({\bf x-s})$$
and 
$$\exists\, {\bf s_0}\in S, {\rm and}\, \p\succ p_t\,  {\rm prime\, ideal}, {\rm such\, that }\, \p|({\bf x-s_0}).$$
Combine the two: It says no prime with an underlying rational prime  $p_t$ divides all coordinates of $({\bf x-s_0})$ and there exists some prime $\p$ lying over $q$ and $q>p_t$ such that $\p|({\bf x-s_0})$, so that ${\bf x}\in {\bf s_0}+ \p^m$. Because ${\bf x}$ is arbitrary in $V(S)_{\vE_t}\setminus V(S)$, the claim follows.

Now we divide the proof into two subsections, one for averaging along $B[R,\vE]$ and the other for averaging along $C[L,\vE]$.

\subsection{Averaging along $B[R,\vE]$}
We recall the following proposition from \cite[Proposition 5.2]{Barak2} (See also \cite[Lemma 1]{Schmidt counting convex}, \cite[Theorem 2.3]{Widmer}.):
\begin{prop}\label{prop: Schmidt counting convex sets}  For any $n\in \vN$ there is $C>0$ so that the following holds. 
Let $\mathcal{S} \subset \vR^n$ be a set, let $\mathcal{L} \subset \vR^n$ be a
lattice, and let $c>0$ and $ T_0 \geq 1$ be such that:
\begin{enumerate}
\item \label{item: Schmidt 1}
$\mathcal{S}$ is convex and its diameter is bounded above by $T_0$.
\item \label{item: Schmidt 2}
The lattice $\mathcal{L}$ contains $n$ linearly independent
vectors of length at most $T_0$, and $n-1$ linearly independent vectors of length at most $c$.
\end{enumerate} 
Then
\begin{equation}\label{eq: asymptotic Schmidt count}
\left| |\mathcal{S} \cap \mathcal{L}|  -
\frac{\vol(\mathcal{S})}{\mathrm{covol}(\mathcal{L})} \right| \leq C cT_0^{n-1}.
\end{equation} 
\end{prop}

\begin{lemma}\label{lemma on points Ideal in a ball}
Let $I \subset \OO$ be an ideal. Then for $R>N(I)^{1/n}$, we have
\begin{equation}
\begin{split}
 \left||(I\cap B[R,\vE])^m|- \frac{\vol (B[R,\vE])^m}{N(I)^m}\right|   \leq C_1\left(\frac{R}{N(I)^{1/n}}\right)^{mn-1}
\end{split}    
\end{equation}
for every $R\in \mathbb{R}$, where $N(I)$ denotes the norm of $I$ and the constant $C_1$ is independent of $R$ and $I$.
\end{lemma}

\begin{proof} We define the problem in terms of the geometry of numbers by embedding the algebraic integers into the Euclidean space $\mathbb{R}^{mn}$. Let $\alpha_{\vE}: \OO \to \mathbb{Z}^n$
be the coordinate isomorphism mapping an element $x = \sum_{i=1}^n a_i e_i$ to the vector $\mathbf{a} = (a_1, \dots, a_n)$, where $e_i$ are the coordinate basis.
This map extends component-wise to 
$$\alpha_{\vE}^m: \mathcal{O}^m \to \mathbb{R}^{mn}.$$
The image of $I$ under $\alpha_{\vE}$, denoted $\Lambda = \alpha_{\vE}(I)$, is a full-rank sublattice of $\mathbb{Z}^n$.
The index of the ideal corresponds to the determinant of the lattice:
$$[\mathcal{O}: I] = [\mathbb{Z}^n: \Lambda] = \det(\Lambda) = N(I).$$
We are interested in the Cartesian product $I^m$. Let $\mathcal{L} = \alpha_{\vE}^m(I^m) \subset \mathbb{R}^{mn}$. The covolume of this product lattice is:
$$\text{covol}(\mathcal{L}) = (\det \Lambda)^m = N(I)^m.$$
Under the embedding $\alpha_{\mathbb{E}}$, $B[R, \mathbb{E}]$ corresponds to the ellipsoid $\mathcal{S}= \alpha_{\vE}^m(B[R, \mathbb{E}]^m)$ in $\mathbb{R}^{mn}$.
We have
$$|(I \cap B[R, \mathbb{E}])^m|= |\mathcal{L}\cap \mathcal{S}|.$$
To apply the counting argument as in Proposition \ref{prop: Schmidt counting convex sets}, we need to verify conditions 1 and~2. 

Put $a= N(I)^{1/n}$ and consider $\phi: \mathbb{R}^{nm}\to \mathbb{R}^{nm}$ such that $\phi(x)=x/a$. Define 
$$\mathcal{L}'= \phi (\mathcal{L}), \quad {\rm and}\quad \mathcal{S}'= \mathcal{S}/a.$$
Then ${\rm covol}(\mathcal{L}')= {\rm covol}(\mathcal{L})/a^{nm}= N(I)^m/N(I)^m=1$. Also the diameter of $\mathcal{S'}$ is $T_0'= T_0/a$, where $T_0\approx 2R$ was the diameter of $S$. Hence
$$T_0'\asymp \frac{2R}{N(I)^{1/n}}.$$
Counting lattice points is invariant under this scaling: $|\mathcal{L}\cap \mathcal{S}|= |\mathcal{L}'\cap \mathcal{S}'|$. The main term transforms as 
\[\frac{\vol(\mathcal{S})}{N(I)^m}=\frac{\vol(\mathcal{S}')}{{\rm covol}(\mathcal{L}')}= \vol(\mathcal{S}').\]
Proposition \ref{prop: Schmidt counting convex sets} requires a parameter $c$ measuring the size of $mn-1$ short independent vectors of the lattice. Standard results from the geometry of numbers (Minkowski’s second theorem) imply that $\Lambda$ admits $n$ linearly independent vectors whose lengths are $\le C_0 \det(\Lambda)^{1/n} = C_0 N(I)^{1/n}$ for a constant $C_0 = C_0(K)$. Consequently, after scaling by $1/a$, those vectors become bounded by $C_0$ (a constant independent of $I$). Therefore, there exists a uniform constant $C_0$ (depending only on $K$) so that $\mathcal{L}'$ has $mn-1$ independent vectors of length $\le C_0$. In short, the small-vector parameter required by Proposition \ref{prop: Schmidt counting convex sets} is uniformly bounded for $\mathcal{L}'$. By Proposition \ref{prop: Schmidt counting convex sets} 
$$\left||\mathcal L' \cap S'| - \operatorname{vol}(S') \right| \le C_2 \, C_0 \, (T_0')^{mn-1}$$
for some constant $C_2$. By substituting $T_0' \asymp 2R/a = 2R/N(I)^{1/n}$ and absorbing $c$ into the constant, we obtain:
\begin{equation}
 \begin{split}
 \left| |\mathcal L \cap S| - \frac{\operatorname{vol}(S)}{N(I)^m} \right| &= \left| |\mathcal L' \cap S'| - \operatorname{vol}(S') \right|\\
 &\le C_3 \left( \frac{2R}{N(I)^{1/n}} \right)^{mn-1}\\
 &\le C_1 \left( \frac{R}{N(I)^{1/n}} \right)^{mn-1}, 
 \end{split}   
\end{equation}
where $C_1 = C_1(K)$.

\end{proof}

By \eqref{eq. non visible from prime larger than pt}, we obtain, for each $R$
\begin{equation}\label{eq. non visible from prime larger than pt in the box}
(V(S)_{\vE_t}\setminus V(S))\cap B[R, \vE]^m\subseteq\bigcup_{{\bf s}\in S}\bigcup_{\p\in \mathcal{P}: \; CR^n\succ\p\succ p_t} ({\bf s}+ \p^m)\cap B[R, \vE]^m,    
\end{equation}
where $C$ is a positive constant that depends only on the chosen basis and not on $R$.

The restriction $CR^n\succ \p$ arises from the following fact: with respect to a fixed basis $\vE$, the norm of an element of $\OO$ is given by a polynomial of degree $n$ in its coordinate entries. Consequently, every element of $B[R, \vE]$ has norm bounded in absolute value by $CR^n$. Therefore $N(B[R, \vE])\subseteq [-C R^n, C R^n]$. On the other hand, if an element of $B[R, \vE]$ is in $\p$ then its norm is divisible by the rational prime $p$ lying under $\p$. This shows that there cannot exist primes $\p\succ CR^n$ containing a nonzero element of $B[R, \vE]$.

From the preceding inclusion, we obtain the estimate
\begin{equation}\label{eq: for upper bound in complement set in box case}
\begin{split}
\overline {D}_{\vE}(V(S)_{\vE_t}\setminus V(S)) & \leq \limsup_{R\rightarrow \infty}\left|\bigcup_{{\bf s}\in S}\bigcup_{\p\in \mathcal{P}: \; CR^n\succ\p\succ p_t} ({\bf s}+ \p^m)\cap B[R, \vE]^m\right|\cdot{(\vol(B[R,\vE])^{-m})}.
\end{split}
\end{equation}
We claim that for every ${\bf s}$ and $\p$, we have for large $R$
$$|({\bf s} + \p^m)\cap B[R, \vE]^m|\leq c_{\bf s}|( \p\cap B[R, \vE])^m|,$$
where the coefficient $c_{{\bf s}}$ is only depends on the points ${\bf s}$. In fact, for each ${\bf s}\in S$ and prime ideal $\p$, the set ${\bf s +\p^m}$ is a shifted sub-lattice. It follows from Lemma \ref{lemma on points Ideal in a ball} that in one dimension, the number of points of a shifted lattice $s_i +\p$ in a ball $B[R,\vE]$ is
$$|(s_i+\p)\cap  B[R,\vE]|= \frac{\vol(B[R,\vE])}{N(\p)} + O\left(\frac{R}{N(\p)^{1/n}}\right)^{n-1}.$$
For $m$ dimensions, the cardinality is the product of the dimensions:
\begin{equation}
\begin{split}
|({\bf s}+\p^m)\cap B[R,\vE]^m|&= \prod_{i=1}^m\left(\frac{\vol(B[R,\vE])}{N(\p)} + O\left(\frac{R}{N(\p)^{1/n}}\right)^{n-1}\right)\\
&= \frac{\vol(B[R,\vE])^m}{N(\p)^m} + O\left(\frac{R}{N(\p)^{1/n}}\right)^{nm-1}.
\end{split}    
\end{equation}
By choosing a constant $c_s$ to absorb the boundary error for large $R$, we get the claim.
Therefore, we get
\begin{equation}\label{eq: for upper bound in complement set2 box case}
\begin{split}
\overline {D}_{\vE}(V(S)_{\vE_t}\setminus V(S)) 
& \leq |S| C_S\limsup_{R\rightarrow \infty}\sum_{\p\in \mathcal P: \; CR^n\succ \p \succ p_t} \frac{|(\p\cap B[R,\vE])|^{m}}{\vol(B[R,\vE])^{m}},
\end{split}
\end{equation}
where $C_S= \max\{c_{\bf s}\, :\, {\bf s} \in S\}$.
By Lemma~\ref{lemma on points Ideal in a ball}, we have, for $R>N(\p)^{1/n}$,
\[
|(\p\cap B[R,\vE])|
\le \frac{\vol(B[R,\vE])}{N(\p)} + C_1\left(\frac{R}{N(\p)^{1/n}}\right)^{n-1}.
\]
Since $\vol(B[R,\vE])\asymp R^{n}$, dividing by $\vol(B[R,\vE])$ and raising $m$-th power give
\begin{align*}
\frac{|(\p\cap B[R,\vE])|}{\vol(B[R,\vE])}
&\le \frac{1}{N(\p)} + \frac{C'}{R\,N(\p)^{1-1/n}},\\
\frac{|(\p\cap B[R,\vE])|^{m}}{\vol(B[R,\vE])^{m}}
&\le \frac{1}{N(\p)^{m}}+\frac{C''}{R^mN(\p)^{m(1-1/n)}}
\le \frac{C}{N(\p)^{m}}
\end{align*}
for $R>N(\p)^{1/n}$, where $C$ depends only on $\vE$, which can take different value line by line.
Consequently, the summand in \eqref{eq: for upper bound in complement set2 box case} is dominated by $C/N(\p)^m$, independently of $R$.
Since, for each fixed $\p$ we have
\[
\lim_{R\to\infty}\frac{|(\p\cap B[R,\vE])|^{m}}{\vol(B[R,\vE])^{m}}=\frac{1}{N(\p)^m}
\]
and the sum $\sum\frac{1} {N(\p)^m}$ converges for $m\ge 2$, we obtain
\begin{equation}\label{eq:dominate_ratiosum}
\limsup_{R\rightarrow \infty}\sum_{\p\in \mathcal P: \; CR^n\succ \p \succ p_t}
\frac{|(\p\cap B[R,\vE])|^{m}}{\vol(B[R,\vE])^{m}}
\le \sum_{\p\succ p_t}\frac{1}{N(\p)^m}\cdotp
\end{equation}
Combining \eqref{eq: for upper bound in complement set2 box case} and \eqref{eq:dominate_ratiosum},
we get
\begin{equation}\label{eq: for upper bound in complement set2 in terms of N(p)}
\begin{split}
\overline {D}_{\vE}(V(S)_{\vE_t}\setminus V(S)) 
& \leq |S| C_S\sum_{\p \succ p_t}\frac{1}{N(\p)^m}\cdotp
\end{split}
\end{equation}
Using the fact that the sum $\sum\frac{1} {N(\p)^m}$ converges for $m\ge2$, again, we get as $t\to \infty$
$$\overline {D}_{\vE}(V(S)_{\vE_t}\setminus V(S))\to 0.$$

\subsection{Averaging along $C[R,\vE]$}
We first introduce a lemma given by Ferraguti and Micheli.
\begin{lemma}[{\cite[Proposition 3.11]{Ferraguti}}]\label{lemma on points Ideal in a cube}
Let $I\subset \OO$ be an ideal. Then for $R>N(I)^{1/n}$, we have
\begin{equation}
\begin{split}
 \left||(I\cap C[R,\vE])^m|- \frac{\vol( C[R,\vE])^{m}}{N(I)^m}\right|   \leq C_1\left(\frac{2R}{cN(I)^{1/n}}+1\right)^{mn-1},
\end{split}    
\end{equation}
for every $R\in \mathbb{R}$, where the constants $C_1$ and $c$ are independent of $R$ and $I$.
\end{lemma}
\begin{remark}
In this paper, we define $C[R,\vE]$ using closed intervals,
whereas \cite{Ferraguti} uses open intervals.
Although our lemma is not stated in exactly the same form as
\cite[Proposition 3.11]{Ferraguti},
it can be verified by following essentially the same argument.
\end{remark}
We have $N(C[R, \vE])\subseteq [-C R^n, C R^n]$, where  $C$ is a positive constant that depends only on the chosen basis and not on $R$. 
By \eqref{eq. non visible from prime larger than pt}, we obtain for each $R$
\begin{equation}\label{eq. non visible from prime larger than pt in the cubes}
(V(S)_{\vE_t}\setminus V(S))\cap C[R, \vE]^m\subseteq\bigcup_{{\bf s}\in S}\bigcup_{\p\in \mathcal{P}_{\OO}: \; CR^n\succ\p\succ p_t} ({\bf s}+ \p^m)\cap C[R, \vE]^m.    
\end{equation}
By \eqref{eq. non visible from prime larger than pt in the cubes}, we obtain the estimate
\begin{equation}\label{eq: for upper bound in complement set}
\begin{split}
\overline  D_\vE(V(S)_{\vE_t}\setminus V(S)) & \leq \limsup_{R\rightarrow \infty}\left|\bigcup_{{\bf s}\in S}\bigcup_{\p\in \mathcal{P}_{\OO}: \; CR^n\succ\p\succ p_t} ({\bf s}+ \p^m)\cap C[R, \vE]^m\right|\cdot{(2R)^{-nm}}.
\end{split}
\end{equation}
As in the previous subsection, we apply
Lemma~\ref{lemma on points Ideal in a cube}
to estimate $|({\bf s} + \p^m)\cap C[R, \vE]^m|$. For every ${\bf s}$ and $\p$, we have for sufficiently large $R$
$$|({\bf s} + \p^m)\cap C[R, \vE]^m|\leq c_{\bf s}|( \p\cap C[R, \vE])^m|,$$
where the coefficient $c_{{\bf s}}$ depends only on the points ${\bf s}$. Therefore, we get
\begin{equation}\label{eq: for upper bound in complement set2}
\begin{split}
\overline  D_\vE(V(S)_{\vE_t}\setminus V(S)) 
& \leq |S| C_S\limsup_{R\rightarrow \infty}\sum_{\p\in \mathcal P: \; CR^n\succ \p \succ p_t} |\p\cap C[R, \vE]|^m\cdot{(2R)^{-nm}},
\end{split}
\end{equation}
where $C_S= \max\{c_{\bf s}\, :\, {\bf s} \in S\}$.
The right-hand side of \eqref{eq: for upper bound in complement set2} goes to zero as $t\to \infty$.
This follows from the same tail estimate established in the proof of \cite[Theorem 3.7]{Ferraguti} or the previous subsection.
\bibliographystyle{plain}

\end{document}